\documentclass{article}
\usepackage{times}
\usepackage[authoryear]{natbib}
\usepackage{natbib}
\usepackage{authblk}
\usepackage{amsmath}
\usepackage{amssymb}
\usepackage{amsfonts}
\usepackage{multirow}
\usepackage{amsthm}
\usepackage{graphicx}
\usepackage{float}
\usepackage{xcolor}
\usepackage{enumitem}
\usepackage{hyperref}

\newtheorem{theorem}{Theorem}
\newtheorem{lemma}{Lemma}
\title{Empirical likelihood for generalized smoothly trimmed mean }
\date{}
\author[1]{Elina Kresse} 
\author[1]{Emils Silins}
\author[1]{Janis Valeinis \thanks{corresponding author: janis.valeinis@lu.lv}}

\affil[1]{Laboratory of Statistical Research and Data Analysis, Faculty of Physics, Mathematics and Optometry, University of Latvia    {\it(Jelgavas Street 3, LV-1586 Riga, Latvia)}}

\begin{document} 
\maketitle

\begin{abstract}
This paper introduces a new version of the smoothly trimmed mean with a more general version of weights, which can be used as an alternative to the classical trimmed mean. We derive its asymptotic variance and to further investigate its properties we establish the empirical likelihood for the new estimator. As expected from previous theoretical investigations we show in our simulations a clear advantage of the proposed estimator over the classical trimmed mean estimator. Moreover, the empirical likelihood method gives an additional advantage for data generated from contaminated models. For the classical trimmed mean it is generally recommended in practice to use symmetrical 10\% or 20\% trimming. However, if the trimming is done close to data gaps, it can even lead to spurious results, as known from the literature and verified by our simulations. Instead, for practical data examples, we choose the smoothing parameters by an optimality criterion that minimises the variance of the proposed estimators.

\end{abstract}

\maketitle                   






\section{Introduction}

L-estimators, defined as a linear combination of order statistics, are commonly used for robust inference of central tendency. The trimmed mean is one of the most popular L-estimators, dating back to 1963 when it was introduced by John W. Tukey and Donald H. McLaughlin (1963). The idea of trimming is to remove a predetermined proportion of the extreme values and calculate the sample mean based on the remaining observations.  For statistically robust inference, especially for comparing several independent groups, we refer to the ``Guide to robust statistical methods'' by Wilcox (2023) and the classical book ``Introduction to robust estimation and hypothesis testing'', also by Wilcox (2011).

Although the trimmed mean is widely used for statistical inference, it has a significant drawback - if the trimming is done at uniquely defined percentiles, the limiting distribution of the trimmed mean may not be asymptotically normal if the data have a discrete or continuous distribution with gaps (Stigler, 1973). This can lead to invalid confidence intervals or test results. For this reason, Stigler (1973) proposed to use the smoothly trimmed mean estimator and showed its asymptotic normality for almost any distribution of interest (Stigler, 1973).
\par Although the smoothly trimmed mean estimator has been mentioned in the literature since 1967 (Crow and Siddiqui, 1967), it has been studied very little. One of the reasons for this may be the difficulty in finding the asymptotic variance and its estimator. In the case of the smoothly trimmed mean, the triangular and trapezoidal weight functions have been considered (Stigler, 1973).
\par We introduce a new version of the smoothly trimmed mean with a more general version of the weight function. The jackknife method has been widely used to estimate the variance of the smoothly trimmed mean (Parr \textit{et.al.}, 1982). Although the technique is effective and straightforward, it is computationally expensive. Using the common technique of influence function calculation, we obtained the asymptotic variance of our estimator and proposed its estimator.
\par{ Confidence intervals based on the t-test and normal approximation for both trimmed and smoothly trimmed means can be imprecise when the data distribution is skewed and the sample size is small or medium. An alternative method for constructing confidence intervals is the empirical likelihood method. In 2002, Qin and Tsao introduced the empirical likelihood method for the trimmed mean (Qin and Tsao, 2002). In 2007, Glenn and Zhao (2007) introduced the weighted empirical likelihood method, which is applied to data that are independent but may not be identically distributed. Combining these two ideas, we establish the empirical likelihood method for the smoothly trimmed mean estimator}.
\par We provide simulations for several contaminated distributions. As expected from Stigler's publication (1973), we see that when there is a gap in the distribution caused by some contaminated models, the asymptotic 95\% coverage is violated for the classical trimmed mean estimator. The behaviour of a new estimator is much more consistent - the coverage accuracy is much closer to 0.95. In addition, the empirical likelihood shows better coverage accuracy, especially when using normally distributed data contaminated with outliers generated by the uniform distribution. 
\par A fixed amount of trimming is often recommended in the literature and used in practical applications. For example, Wilcox proposes a symmetric 20\% trimming, mentioning that it is a good choice for general use as the standard error for such an estimator is much smaller than that of an estimator associated with $\alpha = 0.1$ (Wilcox, 2011). It is a common practice to use 10\% as well. However, such an approach for the trimmed mean in some cases may lead to incorrect results. In terms of coverage accuracy our simulation study confirms that for the distributions with gaps choosing the trimming proportion equal to the level of contamination may lead to wrong results. In contrast, in the case of smoothly trimmed mean it is possible to select the trimming proportion close to the contamination. Another approach is to determine optimal trimming empirically. The literature has considered multiple procedures to set the optimal trimming for the trimmed mean (Jaeckel, 1971; Sawilowsky, 2002; L{\'e}ger, 1990; Keselman, 2007). In this paper we will use Jaeckel's idea for the data examples, which determines as optimal the amount of trimming at which an estimator has the smallest estimated asymptotic variance.
\par The rest of this paper is organised as follows. In Section 2 we define the smoothly trimmed mean and introduce its asymptotic variance. In section 3 we establish empirical likelihood for the smoothly trimmed mean. In Section 4 we compare the coverage accuracy and the average length of 95\% normal approximation based and empirical likelihood confidence intervals for the trimmed and the smoothly trimmed mean and construct the confidence intervals for a real dataset. Section 5 contains some concluding remarks and Appendix is devoted for proofs.

\section{Smoothly trimmed mean and its asymptotic variance}
Let  $X_{1},...,X_{n}$ be independent identically distributed random variables with a common distribution function $F$ and let $X_{(1)} \leq X_{(2)} \leq ... \leq X_{(n)}$ denote the order statistics. The $\alpha$-trimmed mean is defined as
\begin{equation*}
 \overline{X}_{\alpha} = \frac{1}{n-2r} \sum_{i=r+1}^{n-r} X_{(i)},
\end{equation*}
where  $0 \leq \alpha < 0.5$ is trimming proportion and  $r=[n\alpha]$ (Barry \textit{et.al.}, 2008, p. 193).

Let $\mu_{\alpha}$ be the asymptotic mean of $\overline{X}_{\alpha}$. Then the asymptotic variance of the $\alpha$-trimmed mean is 
\begin{align*} 
D(\overline{X}_{\alpha}) & = \frac{1}{(1-2\alpha)^2} \Big (\alpha(F^{-1}(\alpha) - \mu_{\alpha})^2 + \int_{F^{-1}(\alpha)}^{F^{-1}(1-\alpha)} (x-\mu_{\alpha})^2dF(x) \\
 & + \alpha(F^{-1}(1-\alpha) - \mu_{\alpha})^2 \Big)
\end{align*}
and its estimator is 
\begin{align*} 
\widehat{D}(\overline{X}_{\alpha}) & = \frac{1}{n(1-2\alpha)^2} \bigg ( \sum_{i=r+1}^{n-r} (X_{(i)} -  \overline{X}_W)^2 + r(X_{(r+1)} -  \overline{X}_W)^2 \\
& + r(X_{(n-r)} -  \overline{X}_W)^2 \bigg), 
\end{align*}
where $\overline{X}_W$ is the Winsorized mean (Wilcox, 2011, p. 60).

The smoothly trimmed mean is defined as (Stigler, 1973)
\begin{equation}
\overline{X}_{ST} = \frac{1}{n}\sum_{i=1}^{n}J \Big(\frac{i}{n+1}\Big)X_{(i)},
\end{equation}
where $J(\cdot)$ is an appropriate weight function. 
The following weight functions have been considered by Stigler (1973):
\begin{equation}
\label{eq:svari1}
J(u) = \left\{ \begin{array}{lll}
\frac{u-\alpha}{0.5-\alpha}, &\mbox {if $\alpha \leq u \leq 0.5$}\\
\frac{1-u-\alpha}{0.5-\alpha}, &\mbox{if $0.5 < u \leq 1-\alpha$}\\
0, & \mbox{otherwise} \end{array} \right.
\end{equation}
and
\begin{equation}
\label{eq:svari2}
J(u) = \left\{ \begin{array}{llll}
\big(u-\frac{\alpha}{2}\big)\frac{2h}{\alpha}, &\mbox{if $\frac{\alpha}{2} \leq u \leq \alpha$}\\
h, & \mbox{if $\alpha \leq u \leq 1-\alpha$}\\
\big(1-\frac{\alpha}{2} - u\big)\frac{2h}{\alpha}, &\mbox{if $1-\alpha \leq u \leq 1- \frac{\alpha}{2}$}\\
0, & \mbox{otherwise} \end{array} \right.,
\end{equation} 
where $h = 2(2-3\alpha)^{-1}$. 

The three upper plots in Figure \ref{fig:fig1} show the weight function (\ref{eq:svari1}) and the three bottom plots deal with the weight function (\ref{eq:svari2}) with different $\alpha$ values. For the weight function (\ref{eq:svari1}) the trimming proportion equals $\alpha$ and the smoothing parameter is constant $0.5$. Whereas the weight function (\ref{eq:svari2}) contains the trimming proportion equal to $\alpha/2$ and the smoothing parameter equal to $\alpha$.

\begin{figure}[!h]
\begin{center}
\includegraphics[scale=0.6]{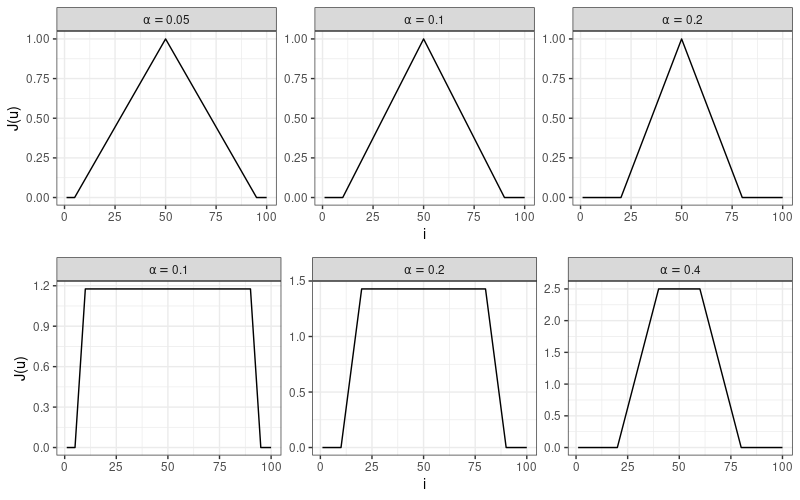}\\
\caption{Weight functions (\ref{eq:svari1}) and (\ref{eq:svari2}).}
\label{fig:fig1}
\end{center}
\end{figure}

As the smoothing parameter in weight functions (\ref{eq:svari1}) and (\ref{eq:svari2}) can not be selected arbitrarily, our idea is to introduce slightly more general estimator by the weight function

\begin{equation}
\label{eq:svari3}
J(u) = \left\{ \begin{array}{llll}
\frac{u-\alpha}{\gamma-\alpha}, &\mbox {if $\alpha \leq u <  \gamma $}\\
1, & \mbox{if $\gamma \leq u \leq 1-\gamma$}\\
\frac{1-u-\alpha}{\gamma-\alpha}, &\mbox{if $1-\gamma < u \leq 1- \alpha $}\\
0, & \mbox{otherwise} \end{array} \right..
\end{equation}
The parameter $\alpha$ defines the trimming proportion and parameter $\gamma$ - the smoothing proportion. In Figure \ref{fig:fig2} some examples of parameter values can be found.
\begin{figure}[!h]
\begin{center}
\includegraphics[scale=0.6]{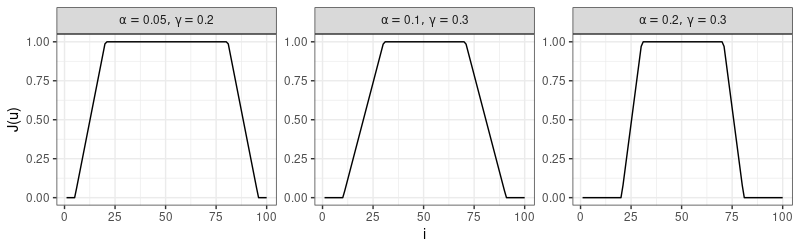}\\
\caption{Weight function (\ref{eq:svari3}).} 
\label{fig:fig2}
\end{center}
\end{figure}
The influence function is used to estimate the asymptotic variance of the smoothly trimmed mean. The relation between the influence function and the asymptotic variance of an estimate is as follows (Wilcox, 2011)
\begin{equation*}
D(T,F) = \int IF^2 (x, T, F)dF(x)
\end{equation*}
and an influence function of L-estimators can be written in the following form (Huber, 2004)
\begin{equation}
\label{eq:infl}
       IF(x, T, F) = \int_{-\infty}^x J(F(y))dy - \int_{-\infty}^{\infty} (1 - F(y)) J(F(y))dy.
\end{equation}
Define 
\begin{align*}
& E_1  = \frac{1}{\gamma - \alpha} \Big( F(x) x  -  \int_{\alpha}^{F(x)} F^{-1}(u)du - \alpha x \Big ), \\
& E_2  = \frac{1}{\gamma - \alpha} \Big( (\gamma - \alpha) F^{-1}(\gamma)- \int_{\alpha}^{\gamma} F^{-1}(u)du  \Big) + x - F^{-1}(\gamma), \\
& E_3  = \frac{1}{\gamma - \alpha}\Big( (\gamma - \alpha) F^{-1}(\gamma) - \int_{\alpha}^{\gamma} F^{-1}(u)du  \Big) + F^{-1}(1-\gamma) - F^{-1}(\gamma) \\
& \hphantom{E_1}  + \frac{1}{\gamma-\alpha}\Big ( (1 - \alpha)x + (\alpha - \gamma) F^{-1}(1-\gamma) - F(x)x +  \int_{1-\gamma}^{F(x)}F^{-1}(u)du \Big), \\
& E_4 = \frac{1}{\gamma - \alpha} \Big( (\gamma - \alpha) F^{-1}(\gamma)- \int_{\alpha}^{\gamma} F^{-1}(u)du  \Big) + F^{-1}(1-\gamma) - F^{-1}(\gamma)\\
& \hphantom{E_1} + \frac{1}{\gamma-\alpha}\Big((\alpha - \gamma) F^{-1}(1-\gamma)+ \int_{1-\gamma}^{1-\alpha}F^{-1}(u)du \Big), \end{align*}
\begin{align*}
& I  \hphantom{1} = \frac{1}{\gamma-\alpha} \Big (  (\gamma +\alpha \gamma - \gamma^2 - \alpha) F^{-1}(\gamma) - (1+\alpha)\int_{\alpha}^{\gamma} F^{-1}(u)du  \\
& \hphantom{E_1} +  2 \int_{\alpha}^{\gamma}u F^{-1}(u)du  + (2-\alpha) \int_{1-\gamma}^{1-\alpha} F^{-1}(u)du + (\alpha\gamma -\gamma^2)F^{-1}(1-\gamma) \\
& \hphantom{E_1} - 2 \int_{1-\gamma}^{1-\alpha} u F^{-1}(u)du  \Big ) +  \int_{\gamma}^{1-\gamma} F^{-1}(u)du +  F^{-1}(1-\gamma)\gamma + F^{-1}(\gamma)\gamma \\
& \hphantom{E_1} - F^{-1}(\gamma).
\end{align*}

\begin{theorem}
Let the weight function for the smoothly trimmed mean be defined as (\ref{eq:svari3}), then the influence function for the smoothly trimmed mean is
\begin{equation*}
\begin{aligned}
IF(x, T, F) = \left\{ \begin{array}{lllll}
-I, & x < F^{-1}(\alpha) \\
E_1 -I, & F^{-1}(\alpha) \leq x < F^{-1}(\gamma) \\
E_2 -I,  &  F^{-1}(\gamma) \leq x \leq F^{-1}(1-\gamma) \\
\begin{aligned}
E_3 -I \end{aligned},& F^{-1}(1-\gamma) < x \leq F^{-1}(1-\alpha)\\
\begin{aligned}
E_4 -I \end{aligned}, & x > F^{-1}(1-\alpha) \end{array} \right.
\end{aligned}
\end{equation*}
\label{theorem1}
and the asymptotic variance is
\begin{equation}
\label{stm_var}
\begin{aligned}
    D(\overline{X}_{ST}) & = \alpha I^2 + \int^{F^{-1}(\gamma)}_{F^{-1}(\alpha)} \Big( E_1 -I \Big )^2 dF(x) + \int^{F^{-1}(1-\gamma)}_{F^{-1}(\gamma)} \Big( E_2 -I \Big )^2 dF(x) \\
    & + \int^{F^{-1}(1-\alpha)}_{F^{-1}(1-\gamma)} \Big( E_3 -I \Big )^2 dF(x) +
    \alpha (E_4 - I)^2.
\end{aligned}
\end{equation}
\end{theorem}
\begin{proof}
The proof of Theorem \ref{theorem1} is given in the Appendix. 
\end{proof}

To estimate the asymptotic variance of the smoothly trimmed mean, define
\begin{align*}
& \widehat{E_1}  = \frac{1}{m - r}  \Big ( (i - r )X_{(i)} - \sum_{k=r+1}^{i} X_{(k)}  \Big), \\
& \widehat{E_2}  = \frac{1}{m - r} \Big ( (m -r) X_{(m+1)} - \sum_{k=r+1}^{m} X_{(k)} \Big) + X_{(i)} - X_{(m+1)}, \\
& \widehat{E_3}  =  \frac{1}{m - r} \Big ( (m - r)X_{(m+1)} - \sum_{k=r+1}^{m} X_{(k)} \Big) + X_{(n-m)} - X_{(m+1)} \\
& \hphantom{\widehat{E_1}} + \frac{1}{m-r} \Big( (n - r) X_{(i)} + (r  - m)X_{(n-m)} - i X_{(i)} + \sum_{k=n-m+1}^{i}X_{(k)} \Big) , \\
& \widehat{E_4} =  \frac{1}{m - r} \Big ( (m - r)X_{(m+1)} - \sum_{k=r+1}^{m} X_{(k)} \Big) + X_{(n-m)} - X_{(m+1)} \\
 & \hphantom{\widehat{E_1}} + \frac{1}{m-r} \Big( (r  - m)X_{(n-m)} + \sum_{k=n-m+1}^{n-r}X_{(k)} \Big), 
 \end{align*} 
 \begin{align*}
& \widehat{I}  \hphantom{1} = \frac{1}{m-r} \Bigg ( \Big(m + \frac{mr}{n} - r - \frac{m^2}{n} \Big) X_{(m+1)} - \Big(1 + \frac{r}{n}\Big) \sum_{i=r+1}^{m} X_{(i)} \\
& \hphantom{\widehat{E_1}} + \frac{2}{n} \sum_{i=r+1}^{m} i X_{(i)}  + \Big(2 - \frac{r}{n}\Big) \sum_{i=n-m+1}^{n-r} X_{(i)} + \Big(\frac{rm - m^2}{n}\Big) X_{(n-m)} \\
& \hphantom{\widehat{E_1}} - \frac{2}{n}  \sum_{i=n-m+1}^{n-r} i X_{(i)} \Bigg ) + \frac{1}{n} \Bigg ( \sum_{i=m+1}^{n-m} X_{(i)} + m X_{(n-m)} + (m -n ) X_{(m+1)} \Bigg ),
\end{align*} 
where  $r = [\alpha n]$ and $m = [\gamma n]$.

The estimator of the asymptotic variance of the smoothly trimmed mean is
\begin{equation}
\begin{aligned}
\label{stm_var_estim}
  \widehat{D}(\overline{X}_{ST}) & = \frac{1}{n^2} \Bigg (K^2 \Big( r \widehat{I}^2  + \sum_{i=r+1}^{m} (\widehat{E_1} - \widehat{I})^2 + \sum_{i=m+1}^{n-m} (\widehat{E_2} - \widehat{I})^2 \\
  & + \sum_{i=n-m+1}^{n-r} (\widehat{E_3} - \widehat{I})^2 + r (\widehat{E_4} - \widehat{I})^2 \Big) \Bigg ),
\end{aligned} 
\end{equation} where $K = \frac{n}{\sum_{i=1}^{n} J \big(\frac{i}{n+1}\big)}$. 

Table \ref{tab:1} compares the jackknife estimate of variance and the asymptotic variance estimates of the smoothly trimmed mean, based on 10,000 repetitions. We can see that both variance estimation methods give similar results. However, the computational time for asymptotic variance is much faster than that of jackknife variance. Column \textit{time ratio} shows the time ratio between the jackknife variance estimation and the asymptotic variance estimation. For example, estimating the asymptotic variance for the smoothly trimmed mean with parameters $\alpha = 0.1$ and $\gamma = 0.2$ of $0.9N(0,1) + 0.1N(0,25)$ sample of size $N = 50$ is 26.85 times faster than estimating the variance by the jackknife method (our computations were done on computer with CPU 2.20 GHz and memory 4.0GB).

\begin{table}[H]
\begin{center}
\caption{Average jackknife and asymptotic variance estimates of the smoothly trimmed mean based on 10,000 replications for two contaminated distributions.}
\label{tab:1} 
{\scriptsize
\begin{tabular}{c c c |c c c | ccc}
\hline
\multicolumn{3}{c}{\multirow{2}{*}{}} & \multicolumn{3}{c}{\multirow{2}{*}{0.9$N$(0,1) + 0.1$N$(0,25)}} & \multicolumn{3}{c}{\multirow{2}{*}{0.1$N$(-10,1) + 0.8$N$(0,1)+ 0.1$N$(10,1)}} \\
\multicolumn{9}{c}{} \\
\hline
 & $\alpha$ & $\gamma$ &  \hfil jack.var &  \hfil asym.var &  \hfil time ratio & \hfil jack.var & \hfil asym.var &\hfil  time ratio \\ 
 \hline
\multirow{4}{*} {$n=20$} & 0.05 & 0.10 & \hfil 1.07830 &  \hfil 1.06692 & \hfil 13.77 & \hfil 1.04319 & \hfil 1.03295 & \hfil 14.23 \\ 
 & 0.10 & 0.20 & \hfil 0.19703 & \hfil 0.19296 & \hfil 13.60 & \hfil 0.64366 & \hfil 0.63049 & \hfil 13.64 \\ 
 & 0.10 & 0.30 & \hfil 0.12821 & \hfil 0.12540 & \hfil 13.50 & \hfil 0.41261 & \hfil 0.40361 & \hfil 13.67 \\ 
 & 0.20 & 0.30 & \hfil 0.08348 & \hfil 0.08126 & \hfil 13.45 & \hfil 0.24063 & \hfil 0.23420 & \hfil 13.71 \\ 
 \hline
\multirow{4}{*} {$n=50$} & 0.05 & 0.10 & \hfil 0.14427 & \hfil 0.17159 & \hfil 27.50 & \hfil 0.40016 & \hfil 0.40864 & \hfil 27.30 \\ 
   & 0.10 & 0.20 & \hfil 0.03318 & \hfil 0.03268 & \hfil 26.85 & \hfil 0.15103 & \hfil 0.14882 & \hfil 25.79 \\ 
   & 0.10 & 0.30 & \hfil 0.03145 & \hfil 0.03097 & \hfil 26.43 & \hfil 0.09330 & \hfil 0.09191 & \hfil 26.33 \\ 
    & 0.20 & 0.30 &\hfil  0.03080 & \hfil 0.03030 & \hfil 27.42 & \hfil 0.04713 & \hfil 0.04638 & \hfil 27.25 \\ 
    \hline
 \multirow{4}{*} {$n=80$} & 0.05 & 0.10 & \hfil 0.05245 & \hfil 0.05196 & \hfil 36.50 & \hfil 0.25105 & \hfil 0.24872 & \hfil 37.85 \\ 
  & 0.10 & 0.20 & \hfil 0.01915 & \hfil 0.01894 & \hfil 37.76 & \hfil 0.07144 & \hfil 0.07070 & \hfil 38.80 \\ 
    & 0.10 & 0.30 & \hfil 0.01892 & \hfil 0.01872 & \hfil 37.39 & \hfil 0.04682 & \hfil 0.04633 & \hfil 37.33 \\ 
   & 0.20 & 0.30 & \hfil 0.01917 & \hfil 0.01896 & \hfil 37.96 & \hfil 0.02676 & \hfil 0.02646 & \hfil 38.62 \\ 
\end{tabular}}
\end{center}
\end{table}

\section{Empirical likelihood for smoothly trimmed mean}
Let 
\[
\mu_{ST} = \int_{-\infty}^{+\infty}xJ(F(x))dF(x) =  \int_0^1 J(u)F^{-1}(u)du
\]
be the asymptotic mean of $\overline{X}_{ST}$ and the weights $w_i$ are determined by the weight function (\ref{eq:svari3}). The empirical likelihood for the smoothly trimmed mean can be established using the weighted empirical likelihood introduced in (Glenn and Zhao, 2007). The proving technique is similar to the paper (Qin and Tsao, 2001), where empirical likelihood for the classical trimmed mean was introduced.

Let the profile empirical ratio function be defined as
\begin{equation*}
        R(\mu_{ST}) = \sup_{p_i} \Big\{ \prod_{i=r+1}^{n-r} \Big(\frac{p_i}{w_i} \Big)^{mw_i}: \sum_{i=r+1}^{n-r} p_i = 1, \sum_{i=r+1}^{n-r} w_i = 1, \sum_{i=r+1}^{n-r} p_iX_i = \mu_{ST} \Big\}, 
\end{equation*}
where $\alpha$ is fixed, $r = [n\alpha]$, $m = n-2r$ and $p_i \geq 0$, $w_i \geq 0$.
Using Lagrange multipliers, we get  
\begin{equation}
\label{peq}
    p_i = \frac{w_i}{1+\lambda W_{(i)}},
\end{equation}
where $W_{(i)} = X_{(i)} - \mu_{ST}$ and  $\lambda$ is the root of
\begin{equation}
\label{lameq}
    \sum_{i=r+1}^{n-r} \frac{w_i W_{(i)}}{1+\lambda W_{(i)}} =0.  
\end{equation}
\noindent Empirical likelihood ratio at $\mu_{ST}$ is
\begin{equation*}
    R(\mu_{ST}) = \prod_{i=r+1}^{n-r} \Big(\frac{p_i}{w_i} \Big)^{mw_i} = \prod_{i=r+1}^{n-r} \Big( \frac{1}{1+\lambda W_{(i)}} \Big )^{m w_i}
\end{equation*}
and the respective test statistic is defined as
\begin{equation}
\label{emp1}
    l(\mu_{ST}) = -2\log R(\mu_{ST}) = 2 \sum_{i=r+1}^{n-r} m w_i \log(1 +\lambda W_{(i)}).
\end{equation}
\begin{theorem}
\label{teo2}
Let $X_1, \ldots, X_n$ be iid random variables with a common cdf $F$. If 
\begin{enumerate}[label=(\roman*)]
    \item $F$ is continuous,
    \item $F'(\xi_{\alpha}) >0$ and $F'(\xi_{1-\alpha}) >0$,
    \item $J(u)$ is bounded and continuous
\end{enumerate} then
\begin{equation}\label{statistic}
    al(\mu_{ST}) \xrightarrow{d} \chi_1^2,
\end{equation}
where
\begin{equation*}
\begin{aligned}
   & a = \frac{\sigma_{ST}^2}{(1-2\alpha) D(\overline{X}_{ST}) n} , \quad \sigma_{ST}^2 =\int_{\xi_{\alpha}}^{\xi_{1-\alpha}} J(u)x^2 dF(x) - \mu_{ST}^2 
\end{aligned}    
\end{equation*}
with $\xi_p = F^{-1}(p)$ and $D(\overline{X}_{ST})$ defined in (\ref{stm_var}).
\end{theorem}
\begin{proof}
The proof of Theorem \ref{teo2} is given in the Appendix. 
\end{proof}
The estimator of scaling constant $a$ is
\begin{equation*}
    \hat{a} = \frac{\widehat{\sigma}_{ST}^2}{(1-2\alpha) \widehat{D}(\overline{X}_{ST}) n },
\end{equation*}
where 
\begin{equation*}
    \widehat{\sigma}_{ST}^2 = \sum_{i = r+1}^{n-r} w_i (X_{(i)} - \overline{X}_{ST})^2
\end{equation*} \\
and $\widehat{D}(\mu_{ST})$ is defined in (\ref{stm_var_estim}). A $100(1-\alpha)$\% empirical likelihood confidence interval is formed by taking those values $\mu_{ST}$ for which $al(\mu_{ST}) \leq \chi_{(1-\alpha),1}^2$.


\section{Simulation study}

To test the asymptotic behavior of the methods proposed in this paper, we perform a simulation study analyzing the empirical coverage accuracy. For this purpose, we generate data from the mixture distribution functions $0.1N(-10,1) + 0.8N(0,1) + 0.1N(10,1)$ and $0.8N(0,1) + 0.2N(0,25)$. Samples generated from these distribution functions will have gaps at quantiles $0.1$ and $0.9$.

\par Figures \ref{fig:fig3} and \ref{fig:fig5} show the simulated results of the coverage accuracy and the average length of the $95\%$ empirical likelihood and normal approximation based confidence intervals for the trimmed and the smoothly trimmed means, based on 10000 replicates. Data were generated from two symmetric distributions: $0.1N(-10,1) + 0.8N(0, 1) + 0.1N(10,1)$ and $0.8N(0,1) + 0.2N(0, 25)$. The first contaminated distribution contains outliers and the second also contains outliers but without visible gaps. Both selected distributions contain 20\% of contamination. $STM_{\gamma}$ denotes the smoothly trimmed mean estimator and $TM$ denotes the trimmed mean estimator. The trimming proportion $\alpha$ is displayed on the $x$-axis, while the $\gamma$ parameter is set as a constant. 
\par When the sample size is small, the simulations clearly show that the empirical coverage accuracy is not as accurate when the trimming proportion is close to the contamination level. As the sample size increases, the coverage accuracy of the confidence intervals for the trimmed mean remains well above 95\%. However, the coverage accuracy of the confidence interval for the smoothly trimmed mean is much closer to 95\%, and it can be observed that the higher the parameter $\gamma$, the closer the coverage accuracy is to 95\%.
\par Now consider the distribution $0.1N(-10,1) + 0.8N(0,1) + 0.1N(10,1)$. From the figure \ref{fig:fig5} we can see that in the case of a symmetrical 5\% or 15\% trimming, the coverage accuracy is close to 95\% for both the trimmed and smoothly trimmed means. However, the average interval length is much higher for the trimmed mean. The empirical likelihood method for smoothly trimmed mean performs marginally better in the case of $10\%$ symmetrical trimming. Furthermore, similar to Figure \ref{fig:fig3}, we can see that the smoothly trimmed mean has a significantly better coverage accuracy for the 10\% symmetrical trimming. Table 2 shows that the trimmed mean does not converge to the normal distribution for the 10\% symmetrical trimming, while the smoothly trimmed mean converges better to the normal distribution $0.95$ quantile ($z_{0.95} \approx 1.96$) with larger sample sizes. 

\newpage
\begin{figure}[h!]
\begin{center}
\includegraphics[scale=0.5]{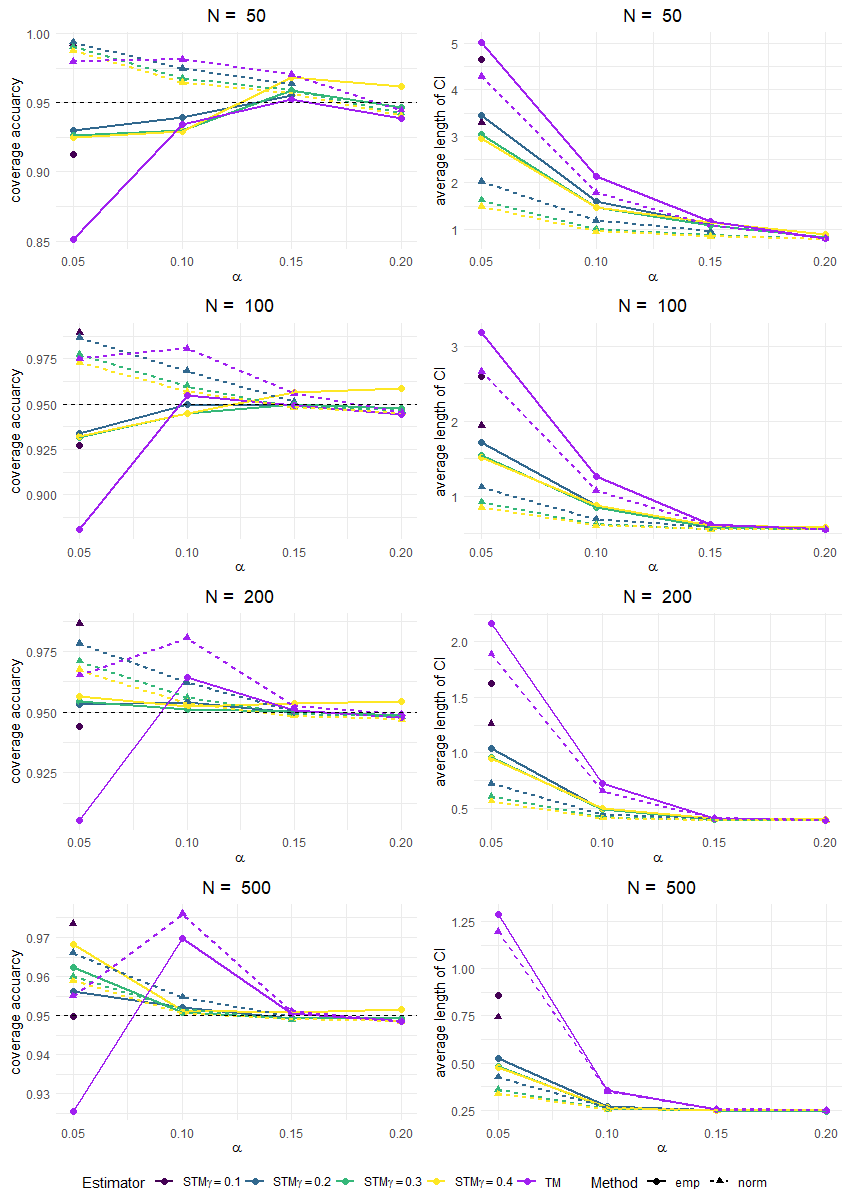}
\caption{Coverage accuracy and average length of 95\% normal approximation and empirical likelihood based confidence intervals, with samples generated from the following mixture distribution function $0.8N(0,1) + 0.2N(0, 25)$, based on 10000 replications.}
\label{fig:fig3}
\end{center}
\end{figure}
\newpage
\begin{figure}[h!]
\begin{center}
\includegraphics[scale=0.5]{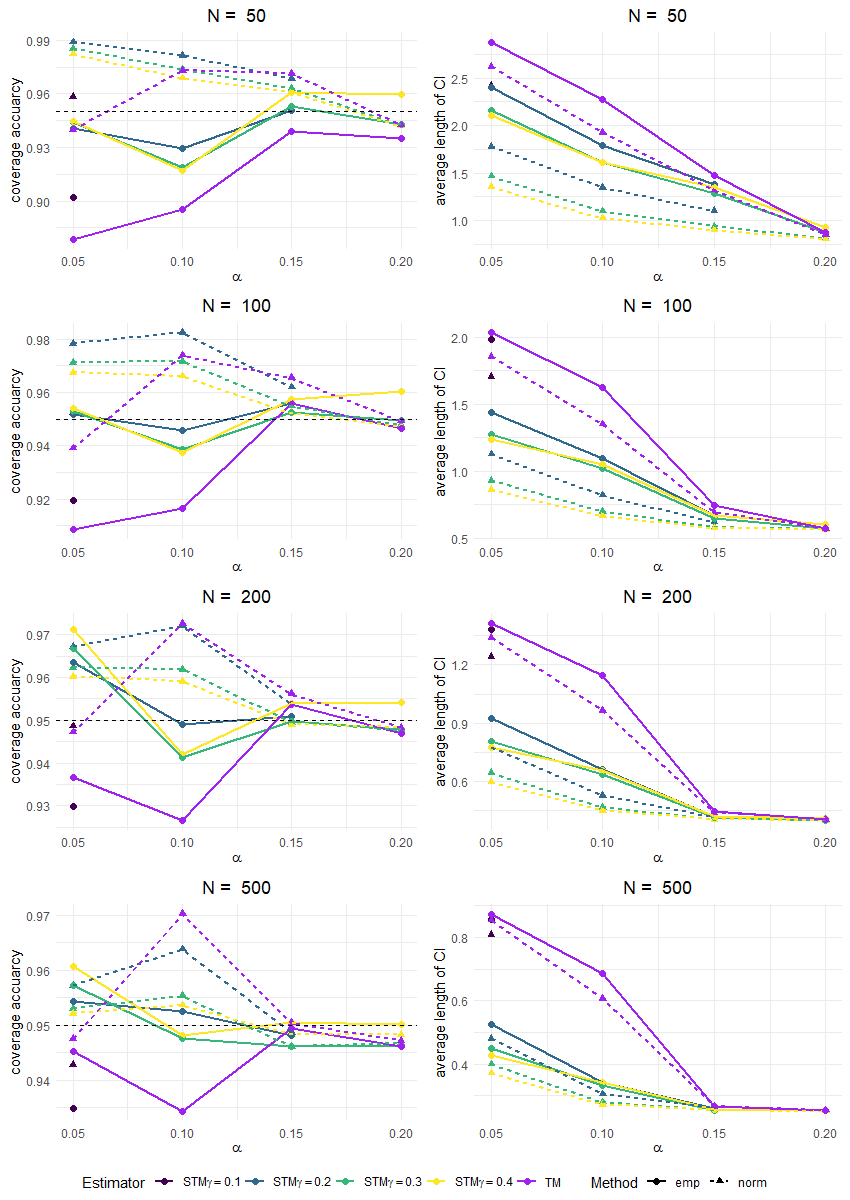}
\caption{Coverage accuracy and average length of 95\% normal approximation and empirical likelihood based confidence intervals, with samples generated from the following mixture distribution function $0.1N(-10,1) + 0.8N(0, 1) + 0.1N(10,1)$, based on 10000 replications. }
\label{fig:fig5}
\end{center}
\end{figure}

\newpage
\begin{table}[h!]
\caption{Empirical $0.95$ quantiles of the smoothly trimmed mean (STM) and the trimmed mean (TM) statistics based on normal approximation. Samples are generated from the mixture distribution $0.1N(-10,1) + 0.8N(0, 1) + 0.1N(10,1)$, based on $10'000$ replications with the true value equal to $1.96$.}
\label{tab:2} 
\centering
{\scriptsize
\begin{tabular}{c c c |c c }
\hline
$n$& $\alpha$ & $\gamma$ &  \hfil STM &  \hfil TM  \\ 
 \hline
                \multirow{4}{*}{100}&    \multirow{4}{*}{0.05} & 0.10 & \hfil 1.9513 & \hfil 2.0742 \\ 
                &           & 0.20 & \hfil 1.7332 & \hfil- \\ 
                 &        & 0.30 & \hfil 1.7915 & \hfil - \\ 
                 &         & 0.40 & \hfil 1.8322 & \hfil -  \\ 
 \hline
                 \multirow{3}{*}{100}&         \multirow{3}{*}{0.10} & 0.20 & \hfil 1.6798 & \hfil1.7456 \\ 
                 &          & 0.30 & \hfil 1.7928 & \hfil -  \\ 
                 &        & 0.40 & \hfil 1.8422 & \hfil -   \\
 \hline
                 \multirow{3}{*}{100}&         \multirow{3}{*}{0.15} & 0.20 & \hfil1.8866 & \hfil1.8579    \\ 
                 &          & 0.30 & \hfil1.9402 & \hfil - \\ 
                 &         & 0.40 & \hfil1.9556 & \hfil -  \\
 \hline
                 \multirow{2}{*}{100}&      \multirow{2}{*}{0.20} & 0.30 & \hfil1.9896 & \hfil1.9873  \\ 
                 &         & 0.40 & \hfil 2.0107 & \hfil -  \\ 

\hline
\hline
      \multirow{4}{*}{500}&   \multirow{4}{*}{0.05} & 0.10 & \hfil 2.0029 & \hfil 1.9549\\ 
                &           & 0.20 & \hfil 1.9138 & \hfil -\\ 
                 &        & 0.30 & \hfil 1.9231 & \hfil - \\ 
                 &         & 0.40 & \hfil 1.9270 & \hfil -  \\ 
 \hline
                  \multirow{3}{*}{500}&         \multirow{3}{*}{0.10} & 0.20 & \hfil 1.8756 & \hfil 1.7861  \\ 
                 &          & 0.30 & \hfil 1.9234 & \hfil -  \\ 
                 &        & 0.40 & \hfil 1.9403 & \hfil -   \\
 \hline
                  \multirow{3}{*}{500}&         \multirow{3}{*}{0.15} & 0.20 & \hfil1.9979 & \hfil 1.9897  \\ 
                 & & 0.30 & \hfil 1.9816 & \hfil - \\ 
                 &        & 0.40 & \hfil 1.9857 & \hfil- \\
 \hline
                  \multirow{2}{*}{500}&      \multirow{2}{*}{0.20} & 0.30 & \hfil1.9775 & \hfil 1.9940  \\ 
                 &        & 0.40 & \hfil 1.9607 & \hfil -  \\ 

\end{tabular}}
\end{table}

\section{Real data examples}
\par In this section we construct $95\%$ confidence intervals for the trimmed and smoothly trimmed means with different trimming and smoothing parameters. The data sets have been selected from the {\tt{$R$ library MASS}}. The data describe measurements of the presence of silicon dioxide ($Si O_2$) in 76 pieces of non-float car window glass and aluminium oxide ($Al_2O_3$) in 70 pieces of float glass. As the data contains outliers, it is advisable to use L-estimators to infer the location parameters. Our aim is to compare the trimmed and smoothly trimmed mean confidence intervals with different trimming and smoothing parameters. We also consider the optimal trimming ratio that minimises the variance of the different estimators.

\begin{figure}[h!]
\begin{center}
\includegraphics[scale=0.35]{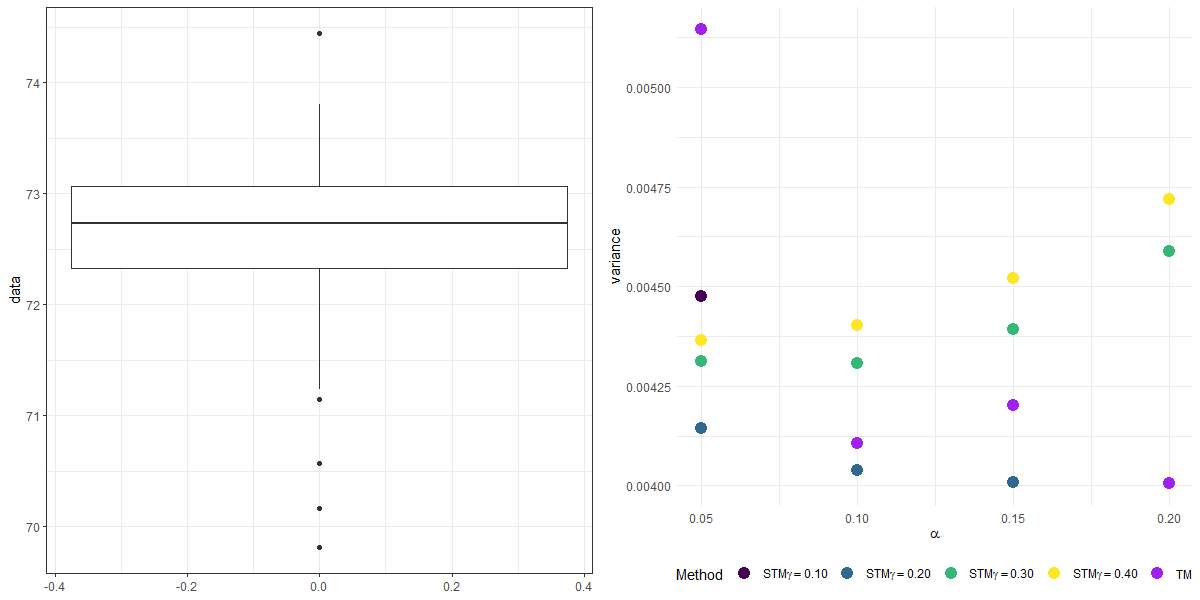}
\caption{Box plot (left panel) and asymptotic variance estimates for the trimmed and the smoothly trimmed mean with different trimming and smoothing parameters for the $Si O_2$ data (right panel).}
\label{fig:figboxplot}
\end{center}
\end{figure}

\par Figure \ref{fig:figboxplot} shows the boxplot of the $Si O_2$ dataset, which contains less than 10\% outliers. The right panel of this figure shows the asymptotic variance estimates for the trimmed and smoothly trimmed mean with different proportions of trimming and smoothing. In figure \ref{fig:figsi} we can see the $95\%$ percentile bootstrap, empirical likelihood and normal approximation based confidence intervals for the trimmed and smoothly trimmed mean, the confidence intervals for the $t$ test and the Wilcoxon test. The confidence intervals marked with dashed lines correspond to the location estimators with the shortest confidence interval.

\begin{figure}[h!]
\begin{center}
\includegraphics[scale=0.5]{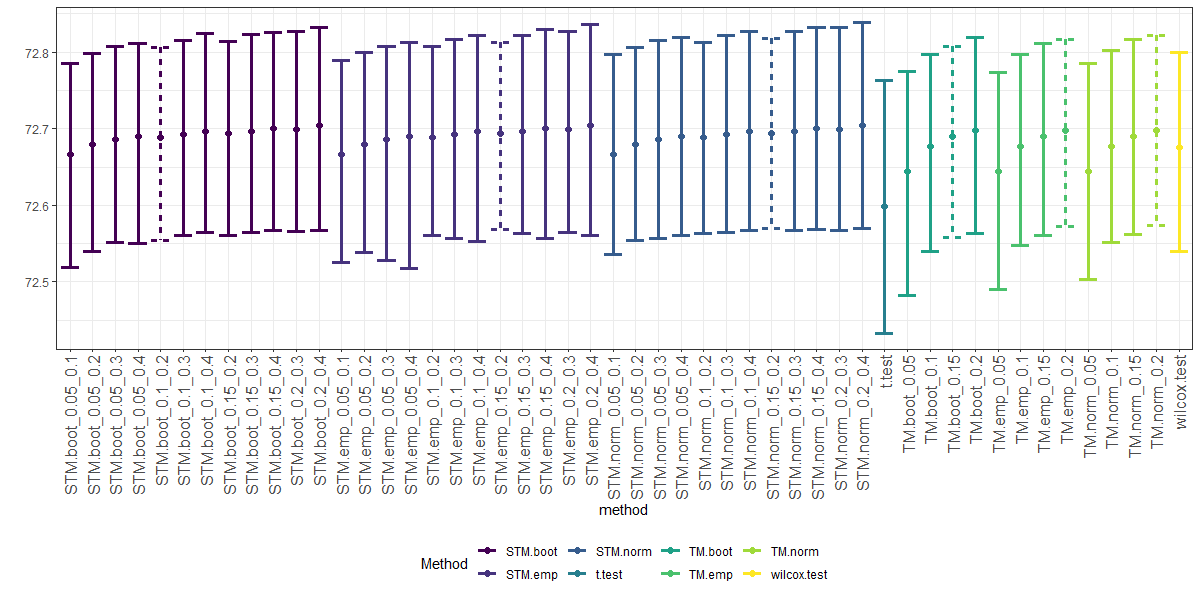}
\caption{$95\%$ bootstrap, normal approximation and empirical likelihood confidence intervals for the trimmed and the smoothly trimmed mean with different trimming and smoothing parameters for the $Si O_2$ data.}
\label{fig:figsi}
\end{center}
\end{figure}
Figure \ref{fig:figsi} shows that the optimal trimming parameter is $\alpha = 0.2$ for the trimmed mean and $\alpha = 0.15$ and $\gamma = 0.2$ for the smoothly trimmed mean. Although the smallest asymptotic variance estimate is for the trimmed mean, smoothly trimmed means with parameters $\alpha = 0.15$ and $\gamma = 0.2$ or $\alpha = 0.1$ and $\gamma = 0.2$ give very close results to the trimmed mean with optimal $\alpha$. To trim as little as possible, one of the versions of the smoothly trimmed mean could be chosen.

 \par Figures \ref{fig:al_var} and \ref{fig:figal} show asymptotic variance estimates and confidence intervals for the trimmed and smoothly trimmed means for the $Al_2O_3$ data. In the figure \ref{fig:al_var} we can see that the lowest asymptotic variance estimates for the trimmed and smoothly trimmed means are at different trimming ratios. The asymptotic variance estimate for the 20\% trimmed mean is larger than that for the 10\% trimmed mean. Therefore, if we had chosen the trimming level recommended in the literature, 20\%, we would not have obtained the estimator with the smallest asymptotic variance estimate. 
 
\par Figure \ref{fig:figal} shows confidence intervals for $Al_2O_3$ data at different trimming and smoothing parameters. We can see that for the smoothly trimmed mean, the shortest confidence interval for the EL method at different trimming and smoothing parameters is shorter than the shortest confidence intervals for the normal approximation-based and bootstrap percentile confidence intervals. In contrast, for a trimmed mean, the bootstrap percentile method gives a shorter confidence interval with a different amount of trimming than the other two methods. Also, the data are not symmetric, so the confidence intervals obtained by the EL method or the bootstrap may be more accurate. 

\begin{figure}[H]
\begin{center}
\includegraphics[scale=0.4]{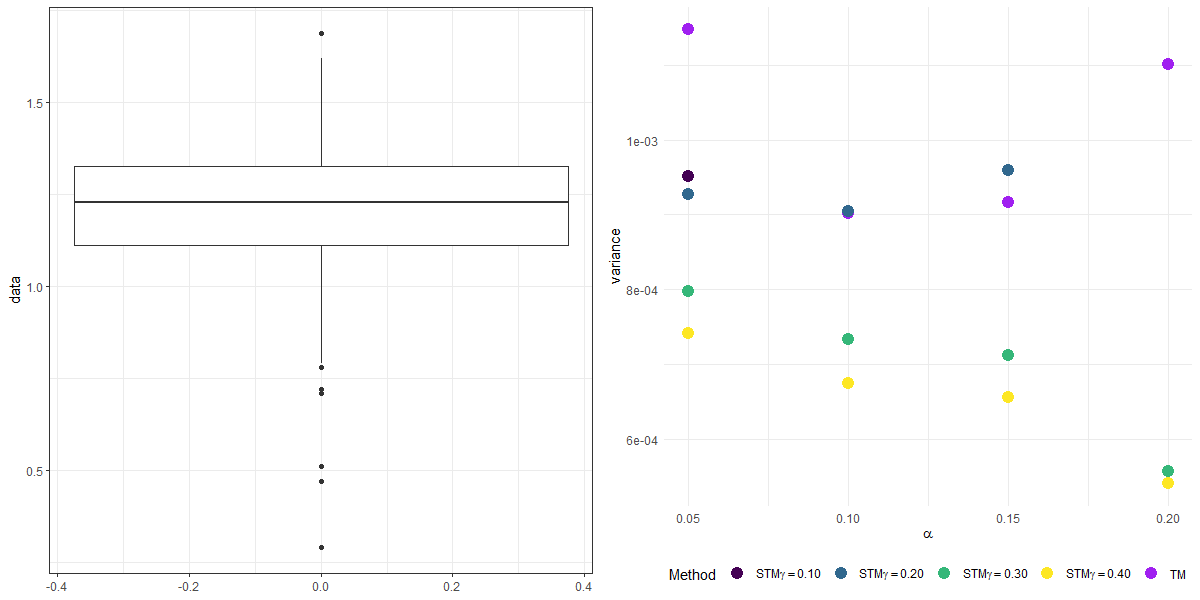}
\caption{Boxplot (left panel) and asymptotic variance estimates for the trimmed and the smoothly trimmed mean with different trimming and smoothing parameters for the $Al_2O_3$ data (right panel).}
\label{fig:al_var}
\end{center}
\end{figure}

\begin{figure}[H]
\begin{center}
\includegraphics[scale=0.42]{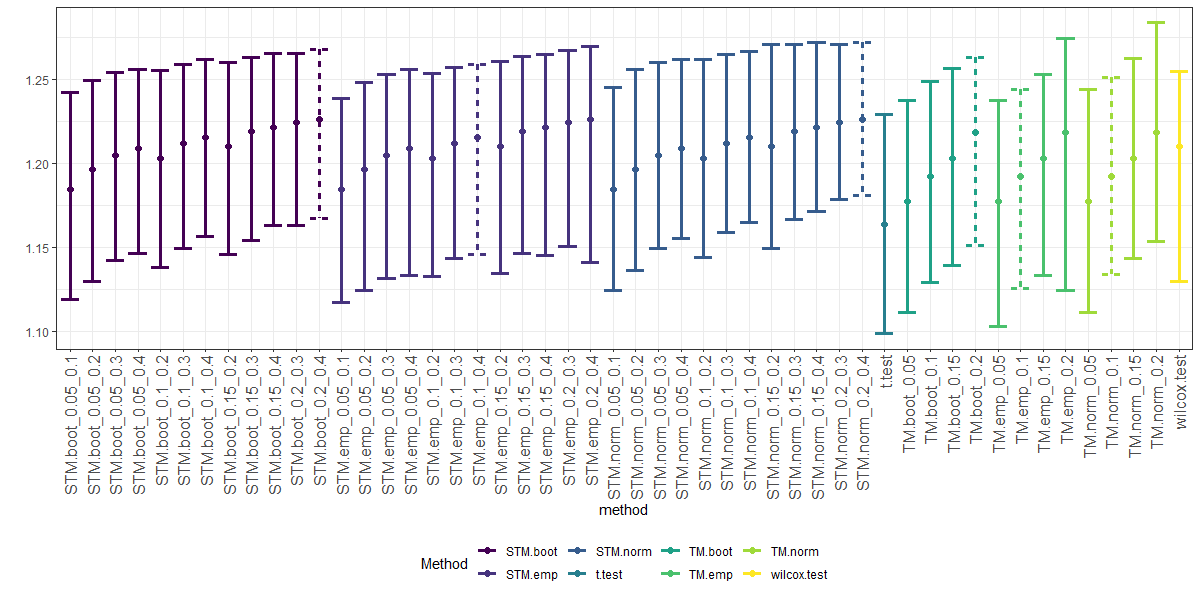}
\caption{$95\%$ bootstrap, normal approximation and empirical likelihood confidence intervals for the trimmed and the smoothly trimmed mean with different trimming and smoothing parameters for the $Al_2O_3$ data.}
\label{fig:figal}
\end{center}
\end{figure}

\section{Conclusions}
A new version of the smoothly trimmed mean has been introduced, which has a clear advantage over the classical trimmed mean. In general, instead of using a fixed trimming rate $\alpha$, one could always choose an additional smoothing parameter $\gamma$ for the smoothly trimmed mean to obtain a more consistent procedure. Furthermore, we derive its asymptotic variance, which is not only a theoretically new result, but also gives a computational advantage. The simulation study confirms that the trimmed mean has poor asymptotic behaviour when the trimming is done close to the gaps of the underlying distribution. The smoothly trimmed mean, on the other hand, overcomes this disadvantage and therefore generally has an advantage over the classical trimmed mean. A smoothly trimmed mean can also be used to reduce the amount of trimming, thus preserving as much of the original structure of the data as possible, while still obtaining an estimator with a small variance.

When using empirical likelihood, the construction of confidence intervals does not require the estimation of variance. It is a nonparametric method, thus we do not need to make distributional assumptions about the data. However, to use the empirical likelihood method one must additionally estimate the scaling constant in \ref{statistic}. The merits and advantages of empirical likelihood in this case remain ambiguous, although our simulations (Figures \ref{fig:fig3} and \ref{fig:fig5}) revealed some benefits. 

\vspace{0.5cm}

\noindent \textbf{Conflict of Interest}

\noindent The authors have declared no conflict of interest.

\section*{Appendix }

\subsection*{A.1. Proof of Theorem 1}
The weight function $J(u)$ for smoothly trimmed mean used in calculation is given by (\ref{eq:svari3}). Let us calculate the first term in (\ref{eq:infl}).

\begin{itemize}
\item If $x < F^{-1}(\alpha)$ then
\begin{equation*}
    \int_{-\infty}^x J(F(y)) dy = 0.
\end{equation*}

\item If $ F^{-1}(\alpha)  <  x < F^{-1}(\gamma)$, then
\begin{equation*}
\begin{aligned}
E_1 & = \int_{-\infty}^x J(F(y)) dy  = \int_{F^{-1} (\alpha)}^x \frac{F(y) - \alpha}{\gamma - \alpha} dy \\
& = \frac{1}{\gamma - \alpha} \int_{F^{-1} (\alpha)}^x (F(y) - \alpha) dy \\
& = \frac{1}{\gamma - \alpha} \Big( F(x) x  -  \int_{\alpha}^{F(x)} F^{-1}(u)du - \alpha x \Big ).
\end{aligned}
\end{equation*}

\item If $F^{-1}(\gamma)  <  x < F^{-1}(1 - \gamma)$, then
\begin{equation*}
\begin{aligned}
E_2 & = \int_{-\infty}^x J(F(y)) dy  = \int_{F^{-1} (\alpha)}^{F^{-1}(\gamma)} \frac{F(y) - \alpha}{\gamma - \alpha} dy + \int_{F^{-1}(\gamma)}^{x} dy \\
& =  \frac{1}{\gamma - \alpha} \Big( (\gamma - \alpha) F^{-1}(\gamma)- \int_{\alpha}^{\gamma} F^{-1}(u)du  \Big) + x - F^{-1}(\gamma).
\end{aligned}
\end{equation*}

\item If $F^{-1}(1-\gamma)  <  x < F^{-1}(1 - \alpha)$, then
\begin{equation*}
\begin{aligned}
E_3 & = \int_{-\infty}^x J(F(y)) dy = \int_{F^{-1} (\alpha)}^{F^{-1}(\gamma)} \frac{F(y) - \alpha}{\gamma - \alpha} dy + \int_{F^{-1}(\gamma)}^{F^{-1}(1-\gamma)} dy \\
& +  \int_{F^{-1}(1-\gamma)}^x \frac{1-F(y) -\alpha}{\gamma - \alpha} dy \\
& = \frac{1}{\gamma - \alpha}\Big( (\gamma - \alpha) F^{-1}(\gamma) - \int_{\alpha}^{\gamma} F^{-1}(u)du  \Big) + F^{-1}(1-\gamma) \\
& - F^{-1}(\gamma)  + \frac{1}{\gamma-\alpha}\Big ( (1 - \alpha)x + (\alpha - \gamma) F^{-1}(1-\gamma) - F(x)x \\
& +  \int_{1-\gamma}^{F(x)}F^{-1}(u)du \Big).
\end{aligned}
\end{equation*}

\item If $x > F^{-1}(1 - \alpha)$, then
\begin{equation*}
\begin{aligned}
E_4 & = \int_{-\infty}^x J(F(y)) dy  = \int_{F^{-1} (\alpha)}^{F^{-1}(\gamma)} \frac{F(y) - \alpha}{\gamma - \alpha} dy + \int_{F^{-1}(\gamma)}^{F^{-1}(1-\gamma)} dy \\
& +  \int_{F^{-1}(1-\gamma)}^{F^{-1}(1-\alpha)} \frac{1-F(y) -\alpha}{\gamma - \alpha} dy \\
& = \frac{1}{\gamma - \alpha} \Big( (\gamma - \alpha) F^{-1}(\gamma)- \int_{\alpha}^{\gamma} F^{-1}(u)du  \Big) + F^{-1}(1-\gamma) \\
& - F^{-1}(\gamma) + \frac{1}{\gamma-\alpha}\Big((\alpha - \gamma) F^{-1}(1-\gamma)+ \int_{1-\gamma}^{1-\alpha}F^{-1}(u)du \Big).
\end{aligned}
\end{equation*}

\end{itemize}
Let us calculate the second term in (\ref{eq:infl}).
\begin{equation}
\label{eq:infl_fun}
\begin{aligned}
    \int_{-\infty}^{\infty}(1-F(y))J(F(y))dy & = \int_{F^{-1}(\alpha)}^{F^{-1}(\gamma)} (1-F(y)) \Big (\frac{F(y)-\alpha}{\gamma-\alpha}\Big ) dy\\
    & + \int_{F^{-1}(\gamma)}^{F^{-1}(1-\gamma)} (1-F(y))dy \\
    & + \int_{F^{-1}(1-\gamma)}^{F^{-1}(1-\alpha)} (1-F(y)) \Big (\frac{1-F(y)-\alpha}{\gamma-\alpha}\Big )dy.
\end{aligned}
\end{equation}
First countable in (\ref{eq:infl_fun}) is
\begin{equation}
\begin{aligned}
\label{eq:term1}
   & \int_{F^{-1}(\alpha)}^{F^{-1}(\gamma)} (1-F(y)) \Big (\frac{F(y)-\alpha}{\gamma-\alpha}\Big )dy \\
   & = \frac{1}{\gamma-\alpha} \int_{F^{-1}(\alpha)}^{F^{-1}(\gamma)} \Big (F(y) - \alpha - F^2(y) + \alpha F(y) \Big ) dy,  
\end{aligned}
\end{equation}
where
\begin{align*}
    & \int_{F^{-1}(\alpha)}^{F^{-1}(\gamma)} (1 + \alpha) F(y)dy = (1 + \alpha) \Big (F^{-1}(\gamma)\gamma - F^{-1}(\alpha)\alpha - \int_{\alpha}^{\gamma} F^{-1}(u)du \Big ), \\
    &  \int_{F^{-1}(\alpha)}^{F^{-1}(\gamma)}\alpha dy = \alpha \Big (F^{-1}(\gamma) - F^{-1}(\alpha) \Big),\\
    &  \int_{F^{-1}(\alpha)}^{F^{-1}(\gamma)} F^2(y) dy = F^{-1}(\alpha)\alpha^2 - F^{-1}(\gamma)\gamma^2 + 2 \int_{\alpha}^{\gamma}u  F^{-1}(u) du. 
\end{align*}
Second countable in (\ref{eq:infl_fun}) is
\begin{equation}
\label{eq:term2}
   \int_{F^{-1}(\gamma)}^{F^{-1}(1-\gamma)} (1-F(y))dy = \int_{\gamma}^{1-\gamma} F^{-1}(u)du +  F^{-1}(1-\gamma)\gamma + F^{-1}(\gamma)\gamma - F^{-1}(\gamma) 
\end{equation}
and third countable is
\begin{equation}
\begin{aligned}
\label{eq:term3}
    \int_{F^{-1}(1-\gamma)}^{F^{-1}(1-\alpha)} & (1-F(y))  \Big (\frac{1-F(y)-\alpha}{\gamma-\alpha}\Big )dy = \\ & = \frac{1}{\gamma-\alpha} \int_{F^{-1}(1-\gamma)}^{F^{-1}(1-\alpha)} \Big (1- 2F(y) - \alpha + F^2(y) + \alpha F(y) \Big ) dy,
    \end{aligned}
\end{equation}
where
\begin{align*}
     & \int_{F^{-1}(1-\gamma)}^{F^{-1}(1-\alpha)} (1 - \alpha) dy = (1 - \alpha) \Big (F^{-1}(1-\gamma) - F^{-1}(1-\alpha) \Big ), \\
     &  \int_{F^{-1}(1-\gamma)}^{F^{-1}(1-\alpha)} (\alpha - 2 )F(y)dy  = (\alpha - 2) \Big ( F^{-1}(1-\alpha)(1-\alpha) \\
     & \hphantom{\int_{F^{-1}(1-\gamma)}^{F^{-1}(1-\alpha)}} - F^{-1}(1-\gamma)(1-\gamma) - \int_{1-\gamma}^{1-\alpha} F^{-1}(u) \Big ) du, \\
     & \int_{F^{-1}(1-\gamma)}^{F^{-1}(1-\alpha)} F^2(y)dy =  F^{-1}(1-\alpha)(1-\alpha)^2 - F^{-1}(1-\gamma)(1-\gamma)^2 \\
     & \hphantom{\int_{F^{-1}(1-\gamma)}^{F^{-1}(1-\alpha)}} - 2\int_{1-\gamma}^{1-\alpha}u F^{-1}(u)du.
\end{align*}

Adding (\ref{eq:term1}), (\ref{eq:term2}) and (\ref{eq:term3}) gives
\begin{equation*}
\begin{aligned}
 I & = \int_{-\infty}^{\infty}(1-F(y))J(F(y))dy \\
& = \frac{1}{\gamma-\alpha} \Big (  (\gamma +\alpha \gamma - \gamma^2 - \alpha) F^{-1}(\gamma) - (1+\alpha)\int_{\alpha}^{\gamma} F^{-1}(u)du  \\
& +  2 \int_{\alpha}^{\gamma}u F^{-1}(u)du  + (2-\alpha) \int_{1-\gamma}^{1-\alpha} F^{-1}(u)du + (\alpha\gamma -\gamma^2)F^{-1}(1-\gamma) \\
& - 2 \int_{1-\gamma}^{1-\alpha} u F^{-1}(u)du  \Big ) +  \int_{\gamma}^{1-\gamma} F^{-1}(u)du +  F^{-1}(1-\gamma)\gamma \\
& + F^{-1}(\gamma)\gamma - F^{-1}(\gamma).
\end{aligned}
\end{equation*}

Combining both terms of (\ref{eq:infl}), we get the influence function for the smoothly trimmed mean in the form
\begin{equation*}
\begin{aligned}
IF(x, T, F) = \left\{ \begin{array}{lllll}
-I, & x < F^{-1}(\alpha) \\
E_1 -I, & F^{-1}(\alpha) < x < F^{-1}(\gamma) \\
E_2 -I,  &  F^{-1}(\gamma) < x < F^{-1}(1-\gamma) \\
\begin{aligned}
E_3 -I \end{aligned},& F^{-1}(1-\gamma) < x < F^{-1}(1-\alpha)\\
\begin{aligned}
E_4 -I \end{aligned}, & x > F^{-1}(1-\alpha) \end{array} \right.
\end{aligned}
\end{equation*}
and the asymptotic variance of the smoothly trimmed mean is 
\begin{equation*}
\begin{aligned}
    D(T_n) & = \alpha I^2 + \int^{F^{-1}(\gamma)}_{F^{-1}(\alpha)} \Big( E_1 -I \Big )^2 dF(x) + \int^{F^{-1}(1-\gamma)}_{F^{-1}(\gamma)} \Big( E_2 -I \Big )^2 dF(x) \\
    & + \int^{F^{-1}(1-\alpha)}_{F^{-1}(1-\gamma)} \Big( E_3 -I \Big )^2 dF(x) +
    \alpha (E_4 - I)^2.
\end{aligned}
\end{equation*}

\subsection*{A.2. Proof of Theorem 2}

The following results are needed to prove the Theorem \ref{teo2}.

\begin{theorem}(Stigler, 1974)\label{teo3}
Assume that $\mathrm{E}(X^2) < \infty$ and $J(u)$ is bounded and continuous a.e. $F^{-1}$ and $D(\overline{X}_{ST}) > 0$ then 
\begin{equation*}
    \frac{\overline{X}_{ST} - \mu_{ST}}{\sqrt{D(\overline{X}_{ST})}} \xrightarrow{d} Z \sim N(0,1).
\end{equation*}
\end{theorem}

\begin{lemma}(Barry \textit{et.al.}, 2008)\label{lema}
    Under the same conditions as in Theorem 2 and 3 
    \begin{equation*}
        \sum_{i = r+1}^{n-r} w_i W_{(i)}^2 \xrightarrow{p} \sigma^2_{ST}.
    \end{equation*}
\end{lemma}

\noindent By Assumption 2 from Theorem \ref{teo2} we have $\max |X_{(i)}| = O_p(1)$. Thus
\begin{equation} \label{as_2}
    \max|W_{(i)}| \leq \max|X_{(i)}| + |\mu_{ST}| = O_p(1).
\end{equation}

\noindent From (Owen, 2001) and the same conditions as in Theorem \ref{teo3}, we know that
\begin{equation*}
    |\lambda| = O_p(n^{-1/2}).
\end{equation*}

\noindent Expressing (\ref{lameq}) in Taylor series we get
\begin{equation}
\label{eq3}
    l(\mu_{ST}) = 2m\sum_{i=r+1}^{n-r} w_i \lambda W_{(i)} - m \sum_{i=1+1}^{n-r} w_i \lambda^2 W_{(i)}^2 + r_n,
\end{equation}
where
\begin{equation}\label{max_W}
    |r_n| \leq B |\lambda^3| \max|W_{(i)}|  \sum_{i=r+1}^{n-r} w_i W_{(i)}^2 = O_p(n^{-3/2}) . 
\end{equation}

\noindent Rewrite (\ref{peq}) in the form
\begin{equation*}
        \sum_{i=r+1}^{n-r} \frac{w_i  W_{(i)} }{1 + \lambda W_{(i)}}  = \sum_{i=r+1}^{n-r} w_i W_{(i)} \Bigg (1 - \lambda W_{(i)} + \frac{ \big(\lambda W_{(i)} \big)^2}{1 + \lambda W_{(i)}} \Bigg ) = 0,
\end{equation*}        
thereby
\begin{equation}
\label{eq4}
    \lambda \sum_{i=r+1}^{n-r} w_i W_{(i)}^2 = \sum_{i=r+1}^{n-r} w_i W_{(i)} + \sum_{i=r+1}^{n-r}  \frac{w_i \lambda^2 W_{(i)}^3}{1 + \lambda W_{(i)}}.
\end{equation}

\noindent By (\ref{max_W}) , (\ref{as_2}) and Lemma \ref{lema} we have
\begin{align*}
    \lambda &=  \Bigg \{ \sum_{i=r+1}^{n-r} w_i W_{(i)}^2 \Bigg \}^{-1} \sum_{i=r+1}^{n-r} w_i W_{(i)} + \sum_{i=r+1}^{n-r}\frac{w_i \lambda^2 W_{(i)}^3}{1 + \lambda W_{(i)}}  \Bigg\{\sum_{i=r+1}^{n-r} w_i W_{(i)}^2\Bigg\}^{-1}\\
    &= \Bigg \{ \sum_{i=r+1}^{n-r} w_i W_{(i)}^2 \Bigg \}^{-1} \sum_{i=r+1}^{n-r} w_i W_{(i)} + O_p(n^{-1}).
\end{align*}

\noindent Multiplying both sides of the expression (\ref{peq}). by $\lambda$, we get
\begin{equation*}
\sum_{i=r+1}^{n-r} \frac{ w_i\lambda  W_{(i)} }{1 + \lambda W_{(i)}} = \sum_{i=r+1}^{n-r} w_i\lambda  W_{(i)} \Bigg (1 - \lambda W_{(i)} + \frac{ \big(\lambda W_{(i)} \big)^2}{1 + \lambda W_{(i)}} \Bigg ) = 0,
\end{equation*}
\noindent therefore 
\begin{equation}
\label{eq2}
    \sum_{i=r+1}^{n-r} w_i \lambda W_{(i)} = \sum_{i=r+1}^{n-r} w_i \lambda^2 W_{(i)}^2 + O_p(n^{-3/2}).
\end{equation}

\noindent Finally, we have
\begin{align*}
    a l(\mu_{ST}) &=a \Bigg ( 2m\sum_{i=r+1}^{n-r} w_i \lambda W_{(i)} - m \sum_{i=1+1}^{n-r} w_i \lambda^2 W_{(i)}^2 \Bigg ) + O_p(n^{-3/2})\\
    & = a  m  \lambda^2 \sum_{i=r+1}^{n-r} w_i W_{(i)}^2 + O_p(n^{-3/2})\\
    & = a m \Bigg ( \lambda \Bigg \{ \sum_{i=r+1}^{n-r} w_i W_{(i)}^2 \Bigg \}^{-1} \sum_{i=r+1}^{n-r} w_i W_{(i)} + O_p(n^{-3/2}) \Bigg ) \sum_{i=r+1}^{n-r} w_i W_{(i)}^2 + O_p(n^{-3/2}) \\
    & = am \Bigg (\lambda  \sum_{i=r+1}^{n-r} w_i W_{(i)} +   O_p(n^{-3/2})  \Bigg ) + O_p(n^{-3/2})\\
    & = am \Bigg ( \sum_{i=r+1}^{n-r} w_i W_{(i)} \Bigg)^2 \Bigg ( \sum_{i=r+1}^{n-r} w_i W_{(i)}^2 \Bigg )^{-1} + O_p(n^{-3/2})\\
    & = \frac{ m \sigma_{ST}^2 }{n (1-2\alpha) D(\overline{X}_{ST})} \Bigg ( \sum_{i=r+1}^{n-r} w_i W_{(i)} \Bigg)^2  \Bigg ( \sum_{i=r+1}^{n-r} w_i W_{(i)}^2 \Bigg )^{-1} + O_p(n^{-3/2}). \\
\end{align*}

\noindent From Theorem \ref{teo3}, Lemma \ref{lema} and Slutsky's lemma it follows that $al(\mu_{ST}) \xrightarrow{d} \chi_1^2$.

\section*{Data availability statement}
R scripts are publicly available in the Github repository \href{https://github.com/LU-SPDAL/publications/tree/main}{\textit{https://github.com/LU-SPDAL/publications/tree/main}}.
\nocite{*}
\bibliographystyle{agsm}
\bibliography{references_1}

\end{document}